\documentclass[10pt]{amsart}

\usepackage[hyperindex]{hyperref}  
\usepackage{amsmath,amssymb,amsrefs}

\def\amalgam#1#2#3{#1\amalg^{\mathfrak{N}_2}_{#3}\!#2}

\numberwithin{equation}{section}

\newtheorem{theorem}[equation]{Theorem}

\theoremstyle{definition}
\newtheorem{definition}[equation]{Definition}

\theoremstyle{remark}

\begin{document}

\title[Embedding $p$-groups of class two in capable and non-capable
  groups]{Embedding groups of class two and prime exponent in capable
  and non-capable groups}
\author{Arturo Magidin} 
\address{Mathematics Dept. University of
  Louisiana--Lafayette, 217 Maxim Doucet Hall, P.O.~Box 41010,
  Lafayette LA 70504-1010} 
\email{magidin@member.ams.org}
\subjclass[2000]{20D15} 
\thanks{The author was supported by a grant
  from the Louisiana Board of Regents Support Fund}

\begin{abstract}
We show that if $G$ is any $p$-group of class at most two and
exponent~$p$, then there exist groups $G_1$ and $G_2$ of class two and
exponent~$p$ that contain~$G$, neither of which can be expressed as a
central product, and with $G_1$ capable and $G_2$ not capable. We
provide upper bounds for ${\rm rank}(\smash{G_i}^{\rm ab})$ in terms of ${\rm
rank}(G^{\rm ab})$ in each case.
\end{abstract}

\maketitle

\section{Introduction.}\label{sec:intro}

A group is called \textit{capable} if it is a central factor
group. Capability plays an important role in P.~Hall's scheme of
classifying $p$-groups up to isoclinism~\cite{hallpgroups}, and has
interesting connections to other branches of group theory. The
finitely generated capable abelian groups were classified by
Baer~\cite{baer}. It is not too difficult to determine wether a given
finitely generated group of class two is capable or not by finding its
\textit{epicentre} $Z^*(G)$, the smallest subgroup of $G$ such that
$G/Z^*(G)$ is capable. Computing the epicentre in this setting reduces
to relatively straightfoward computations with finitely generated
abelian groups; see for example~\cite{ellis}*{Thms.~4 and~7}. However,
currently available techniques seem insufficient to give a
classification of capable finitely generated groups of class two along
the lines of Baer's result for the abelian case.

A full classification for the $p$-groups of class two and prime
exponent seems a modest and possibly attainable goal; some purely
numerical necessary conditions~\cite{heinnikolova}*{Thm.~1} and
sufficient ones~\cite{expp}*{Thm.~5.26} are known, and a number of
results allow us to reduce the problem to a restricted subclass. If we
let $G$ be a $p$-group of class at most two and odd prime
exponent, it is not hard to show (for example, using
\cite{beyl}*{Prop.~6.2}) that if $G$ is a nontrivial direct product,
then $G$ is capable if and only if each direct factor is either
capable or nontrivial cyclic. We can write $G$ as $G=K\times C_p^n$
where $K$ is a group that satisfies $Z(K)=K'$ and $C_p$ is cyclic of
order~$p$, and so $G$ is capable if and only if $K$ is nontrivial
capable, or $K$ is trivial and $n>1$. If $G$ can be decomposed as a
nontrivial central product, $G=CD$ with $[D,C]=\{e\}$ and $\{e\}\neq
[C,C]\cap [D,D]$ (where $e$ is the identity of the group), then $G$ is
not capable \cite{heinnikolova}*{Prop.~1}.  Ellis proved that if
$\{x_1,\ldots,x_n\}$ is a transversal for $G/Z(G)$, and the nontrivial
commutators $[x_j,x_i]$, $1\leq i<j\leq n$ form a basis for~$[G,G]$,
then $G$ is capable~\cite{ellis}*{Prop.~9}. We are then reduced to
considering groups in a restricted class; we give it a name for future
reference:

\begin{definition} We will denote by $\mathcal{R}_p$ the class of all
  $p$-groups $G$ of odd prime exponent~$p$ that
  cannot be decomposed into a nontrivial central product, with
  $Z(G)=[G,G]$, and such that if
  $\{x_1,\ldots,x_n\}$ is a transversal for $G/Z(G)$, then there is a
  nontrivial relation among the nontrivial commutators of the form
  $[x_j,x_i]$, $1\leq i<j\leq n$.
\end{definition}

Unfortunately, the situation appears to be far from straightforward
once we reach this point. In particular, as the two main results in
this paper show, to determine whether such a $G$ is capable we need a
``holistic'' examination of~$G$: there are no forbidden-subgroup
criteria for the capability or non-capablity of~$G$. Explicitly, we
show that if $G$ is any group of class at most two and odd prime
exponent~$p$, then~$G$ is contained in groups $G_1$ and $G_2$,
both in $\mathcal{R}_p$, and with $G_1$ capable and~$G_2$ not~capable. 

In Section~\ref{sec:twonilprod} we prove some properties of the
$2$-nilpotent product with amalgamation, which is our main
construction tool. Section~\ref{sec:mainresults} contains our
main results. 

\section{The $2$-nilpotent product with amalgamation.}\label{sec:twonilprod}

Groups will be written multiplicatively, and we will use $e$ to denote
the identity element of~$G$; $C_p$ denotes the cyclic group of order~$p$.

\begin{definition}
Let $A$ and $B$ be nilpotent groups of class at most~$2$. The
\textit{$2$-nilpotent product of~$A$ and $B$} is defined to be the group
$F/[[F,F],F]$, where $F=A*B$ is the free product of $A$ and~$B$.
We denote the
$2$-nilpotent product of $A$ and~$B$ by $\amalgam{A}{B}{}$.
\end{definition}

The $2$-nilpotent product was introduced by Golovin~\cite{metab}, with
a more general definition that applies to any two groups~$A$
and~$B$. If $A$ and $B$ are nilpotent of class at most two, then
their $2$-nilpotent product is their coproduct (in the sense of
Category Theory) within the variety of all groups of class at most
two, hence our choice of notation. The elements of $\amalgam{A}{B}{}$
can be written uniquely as $\alpha\beta\gamma$, with $\alpha\in A$,
$\beta\in B$, and $\gamma\in[B,A]$; multiplication in
$\amalgam{A}{B}{}$ is then given by:
\[ (\alpha_1\beta_1\gamma_1)(\alpha_2\beta_2\gamma_2) =
(\alpha_1\alpha_2)(\beta_1\beta_2)(\gamma_1\gamma_2[\beta_1,\alpha_2]).\]
A theorem of T.~MacHenry~\cite{machenry} shows that $[B,A]\cong B^{\rm
ab}\otimes A^{\rm ab}$ via $[b,a]\mapsto \overline{b}\otimes
\overline{a}$ (where $\overline{x}$ denotes the image of $x$ under the
canonical maps $A\to A^{\rm ab}$ and $B\to B^{\rm ab}$, and $A^{\rm
  ab}\otimes B^{\rm ab}$ is 
the usual tensor product of abelian groups).
Note that $\amalgam{A}{B}{}$ contains isomorphic copies of
$A$ and~$B$ and is generated by these copies, so if $A$ and~$B$ are
both of odd exponent~$n$, then so is $\amalgam{A}{B}{}$. 

\begin{definition} Let $A$ and $B$ be nilpotent groups of class at most~$2$, 
let $H\leq [A,A]$, $K\leq [B,B]$, and let $\varphi\colon H\to K$ be an
isomorphism. The \textit{amalgamated coproduct of $A$ and~$B$ along $\varphi$}
is defined to be the group
\[ \amalgam{A}{B}{\varphi} = 
\frac{\amalgam{A}{B}{}}{\{ h\varphi(h)^{-1}\,|\, h\in
  H\}}.\]
\end{definition}
It is again easy to verify that $\amalgam{A}{B}{\varphi}$ contains
isomorphic copies of $A$ and~$B$ whose intersection is exactly the
image of $H$ identified with the image of~$K$ as indicated
by~$\varphi$. 

Recall that a group $G$ is said to be a \textit{central product} of
subgroups $C$ and $D$ if and only if $G=CD$ and $[C,D]=\{e\}$.
The direct product is thus a special case of the
central product. If in addition we have that $C\cap D=Z(G)$ then we
say that $G$ is the \textit{full central product} of $C$ and~$D$.  The
central product is said to be trivial if $C\subseteq D$ or $D\subseteq
C$. We prove that if $A$ and~$B$ are both nontrivial, then
$\amalgam{A}{B}{\varphi}$ cannot be decomposed as a nontrivial central
product.

\begin{theorem} Let $A$ and $B$ be nontrivial groups of class at
  most~$2$ and odd prime exponent~$p$, let $H\leq [A,A]$, $K\leq [B,B]$, and let
  $\varphi\colon H\to K$ be an isomorphism. If
  $G=\amalgam{A}{B}{\varphi}$, then $Z(G)=[G,G]$ and $G$ cannot be
  expressed as a nontrivial central~product.
\label{thm:amalgamnotcentralprod}
\end{theorem}

\begin{proof} Identify $A$ and $B$ with their images in~$G$. Let
  $\{a_1,\ldots,a_m\}$ be a transversal for $A/[A,A]$, 
  and $\{b_1,\ldots,b_n\}$ a transversal for $B/[B,B]$.
  Since $A$ is of exponent~$p$, every
  element of $A$ can be written uniquely in the form $a_1^{r_1}\cdots
  a_m^{r_m}a'$ with $0\leq r_i < p$ and $a'\in [A,A]$, and similarly 
  for every element of $B$. From the
  construction of the amalgamated coproduct it follows that
  $\{a_1,\ldots,a_m,b_1,\ldots,b_n\}$ is a transversal for $G/[G,G]$,
  and that $[G,G]\cong [A,A][B,B]\times [B,A]$. In particular, the
  commutators $[b_j,a_i]$, $1\leq i\leq m$, $1\leq j\leq n$ form a
  linearly independent subset of $[G,G]$ (viewing the latter as an
  $\mathbb{F}_p$-vector space).

Let $g\in G$; then we may
write $g=a_1^{r_1}\cdots a_m^{r_m}b_1^{s_1}\cdots b_n^{s_n}g'$, with $0\leq
r_i,s_j < p$ and $g'\in[G,G]$.  We first assert that $g\in Z(G)$ if and only if
$r_i=s_j=0$ for all $i$ and~$j$. Indeed, we have:
\[\relax [g,a_i] = 
\prod_{j=1}^m[a_j,a_i]^{r_j}\cdot\prod_{j=1}^m [b_j,a_i]^{s_j}.\]
For this to be trivial, we must have $s_j=0$ for each~$j$.
Symmetrically, computing $[g,b_j]$ we obtain that $r_i=0$ for each
     $i$. Thus, $Z(G)=[G,G]$, as claimed.

Now suppose that $G$ is decomposed as a central product, $G=CD$. We can express
$a_1=c_1d_1$ for some $c_1\in C$, $d_1\in D$. Let
\begin{eqnarray*}
c_1 & = & a_1^{r_1}\cdots a_m^{r_m}b_1^{s_1}\cdots b_m^{s_m} c',\\
d_1 & = & a_1^{\rho_1}\cdots a_m^{\rho_m}b_1^{\sigma_1}\cdots
b_m^{\sigma_m}d',
\end{eqnarray*}
with $c',d'\in [G,G]$; since $c_1d_1=a_1$, we have $r_1+\rho_1\equiv
1\pmod{p}$, and $r_i+\rho_i\equiv s_j+\sigma_j\equiv 0\pmod{p}$ for
$2\leq i\leq m$, $1\leq j\leq n$. Since $[d_1,c_1]=e$, and
\[\relax [d_1,c_1] = a''b''\prod_{i=1}^m\prod_{j=1}^n [b_j,a_i]^{\sigma_jr_i
  - s_j\rho_i}, \qquad a''\in[A,A],\ d''\in[B,B],\] we have $0\equiv
\sigma_j r_1-s_j\rho_1\equiv \sigma_jr_1+\sigma_j(1-r_1)\equiv
\sigma_j\equiv -s_j\pmod{p}$. That is, $c_1,d_1\in A[G,G]$. Since
$a_1\notin Z(G)$, either $c_1\notin Z(G)$ or $d_1\notin Z(G)$. If both
hold, then we have $C\subseteq C_G(d_1)\subseteq A[G,G]$ and
$D\subseteq C_G(c_1)\subseteq A[G,G]$, so we conclude that $G=CD\subseteq A[G,G]$,
which is impossible. Thus, exactly one of $c_1$ and~$d_1$ is
noncentral. Without loss of generality say $c_1\notin Z(G)$ and
$d_1\in Z(G)$, so $a_1=c_1d_1\in CZ(G)$.

Since none of $b_1,\ldots,b_n$ commute with $a_1$, and $a_1\in
CZ(G)$, we must have $b_1,\ldots,b_n\in CZ(G)$ as well; that is,
$B\subseteq C[G,G]$. And since $b_1\in C[G,G]$ and none of
$a_1,\ldots,a_m$ commute with $b_1$, we must also have
$a_1,\ldots,a_m\in C[G,G]$. Thus, $G=C[G,G]$, and so $[G,G]=[C,C]$ and
$D\subseteq C$.  Hence the central product decomposition $G=CD$ is
trivial, as claimed.
\end{proof}

We finish this section by describing the epicentre of an amalgamated
coproduct.

\begin{theorem}
Let $A$ and $B$ be nontrivial groups of class at most~$2$ and
odd prime exponent~$p$, let $H\leq [A,A]$, $K\leq [B,B]$, and let
$\varphi\colon H\to K$ 
be an isomorphism. If $G=\amalgam{A}{B}{\varphi}$, then
$Z^*(G) = \{ h\in H\,|\, h\in Z^*(A)\ \mbox{and}\ \varphi(h)\in
Z^*(B)\}$; that is, if we identify $A$ and~$B$ with their images in
$G$, then $Z^*(G) = Z^*(A)\cap Z^*(B)$. 
\label{thm:epiforamalgams}
\end{theorem}

\begin{proof} Let $a_1,\ldots,a_m$ be a transversal for $A/[A,A]$, and
  $b_1,\ldots,b_n$ a transversal for $B/[B,B]$. Following Ellis, for
 $x,y,z\in\{a_i,b_j\,|\, 1\leq i\leq m, 1\leq j\leq n\}$ let 
\[ J(x,y,z) = [x,y]\otimes\overline{z} + [y,z]\otimes\overline{x} +
 [z,x]\otimes\overline{y},\]
and let $S$ be the subgroup of $[G,G]\otimes G^{\rm ab}$ generated by all
such $J(x,y,z)$. By
\cite{ellis}*{Theorem~7}, an element $g\in Z(G)=[G,G]$ lies in
$Z^*(G)$ if and only if $g\otimes w\in S$ for all $w\in\{a_i,b_j\}$. 
Since $[G,G]\cong [A,A][B,B]\times (B^{\rm ab}\otimes A^{\rm ab})$,
$G^{\rm ab}\cong A^{\rm ab}\times B^{\rm ab}$, and all factors are
elementary abelian $p$-groups, we have that $[G,G]\otimes G^{\rm ab}$ is
iso\-morphic~to 
\[ \Bigl(C\otimes (A^{\rm ab}\times
B^{\rm ab})\Bigr) \oplus \left(\bigoplus_{i,k=1}^m\bigoplus_{j=1}^n
\langle[b_j,a_i]\otimes\overline{a_k}\rangle \right) \oplus
\left(\bigoplus_{j,k=1}^{n}\bigoplus_{i=1}^m
\langle[b_j,a_i]\otimes\overline{b_k}\rangle\right),\] where $C$ is
the central product of $[A,A]$ and $[B,B]$ obtained by identifying $H$
with~$K$ along~$\varphi$. An easy computation shows that none of the
elements $J(x,y,z)$ has a nontrivial $[b_j,a_i]\otimes \overline{b_j}$
component. Thus, if $g\in [G,G]$ has a nontrivial $[B,A]$ component, say
$[b_j,a_i]$, then it follows that $g\otimes \overline{b_j}$ does not
lie in~$S$, so $g$ is not in $Z^*(G)$. Thus, $Z^*(G)\subseteq
[A,A][B,B]$. 

Consider the elements $J(x,y,z)$ in which at least one of $x$, $y$,
or~$z$ is equal to~$b_1$. Unless the other two are in~$B$, the
generators include nontrivial $[B,A]\otimes(B^{\rm ab}\times A^{\rm
ab})$ components that do not occur in any other generator, and occur
in pairs.  It is
straightforward then that if $g\in [G,G]$ lies in $Z^*(G)$, since
$g\otimes\overline{b_1}$ must lie in $S$ we have that $g$ can be
expressed in terms of commutators of $b_1,\ldots,b_m$; that is, $g\in
[B,B]$. By a symmetric argument considering $a_1$ instead, we obtain
that if $g\in Z^*(G)$ then $g\in [A,A]$. Thus, $Z^*(G)$ is contained
in $[A,A]\cap[B,B]$; recall that this intersection is equal to the
identified subgroups $H=K$.

If $g\in H$ lies in $Z^*(G)$, then $g\otimes\overline{a_i}\in S$ for
each $i$, and this readily yields that for all $a\in A$,
$g\otimes\overline{a}$ lies in the subgroup of $[A,A]\otimes A^{\rm
ab}$ generated by all $J(a_i,a_j,a_k)$; thus $g\in Z^*(A)$;
symmetrically, since $g\otimes\overline{b_j}\in S$ we obtain that $g$
(considered now as an element of~$K$) lies in $Z^*(B)$, so
$Z^*(G)\subseteq Z^*(A)\cap Z^*(B)$. Conversely, if $h\in H\cap
Z^*(A)$ is such that $\varphi(h)\in Z^*(B)$, then
$h\otimes\overline{a_i}\in \langle J(a_r,a_s,a_t)\rangle$ for all~$i$
and $\varphi(h)\otimes\overline{b_j}\in \langle J(b_r,b_s,b_t)\rangle$
for all~$j$, hence $h=\varphi(h)\in Z^*(G)$, giving the
desired~equality.
\end{proof}

\section{Main results.}\label{sec:mainresults}

We now give the promised results.

\begin{theorem} Let $G$ be any nontrivial group of class at most~$2$ and
 odd prime exponent~$p$. Then there exists a capable group
 $G_1\in\mathcal{R}_p$ that contains~$G$. If $G$ is nonabelian and
 capable, then we may choose $G_1$ so that ${\rm rank}(G_1^{\rm
   ab})\leq {\rm rank}(G^{\rm ab})+2$. Otherwise,
 we may choose $G_1$ such that ${\rm rank}(G_1^{\rm
   ab})\leq {\rm rank}(G^{\rm ab})+3$.
\label{th:subofcapable}
\end{theorem}

\begin{proof} We construct $G_1$ in two steps. If $G$ is nonabelian
  and capable, set $G_0=G$; otherwise, let $G_0 =
  \amalgam{G}{C_p}{}$. To obtain~$G_1$, let $H=\amalgam{C_p}{C_p}{}$,
  and let $\varphi$ be an isomorphism between $[H,H]$ and a nontrivial
  cyclic subgroup of $[G_0,G_0]$. Finally, let
  $G_1=\amalgam{G_0}{H}{\varphi}$. Since $G_0$ is capable, $G_1$ is
  capable; by Theorem~\ref{thm:amalgamnotcentralprod} $G_1$ is not a
  nontrivial central product. The identification of the generator of
  $[H,H]$ with a nontrivial element of $[G_0,G_0]$ guarantees the
  existence of a nontrivial relation among nontrivial commutators of
  any transversal, hence $G_1\in\mathcal{R}_p$, as desired. The rank
  inequality is immediate.
\end{proof}

\begin{theorem} Let $G$ be any nontrivial group of class at most~$2$
  and exponent~$p$. Then there exists a non-capable group
    $G_2\in\mathcal{R}_p$ that contains~$G$. 
   If $G$ is nonabelian, then we may choose $G_2$ such that ${\rm rank}(G_2^{\rm
  ab})\leq {\rm rank}(G^{\rm ab})+6$. If $G$ is abelian, then we may
  choose $G_2$ with ${\rm rank}(G_2^{\rm ab}) \leq {\rm rank}(G^{\rm
  ab})+7$.
\label{th:subofnoncapable}
\end{theorem}

\begin{proof} If $G$ is nonabelian, let $H_1=\amalgam{C_p}{C_p}{}$,
  let $g\in [G,G]$ be nontrivial, and let~$H$ be the 
  central product of $G$ and~$H_1$ 
  identifying $g$ with a generator of $[H_1,H_1]$. Since this is a
  nontrivial central product with $[G,G]\cap[H_1,H_1]\neq\{e\}$, it is
  not capable and $g\in Z^*(H)$ by \cite{heinnikolova}*{Prop.~1}. Now let
  $G_2=\amalgam{H}{E}{\varphi}$, where $E$ is an extraspecial group of
  order $p^5$ and exponent~$p$, and $\varphi$ identifies $g$ with a
  generator of $[E,E]$ (which is in $Z^*(E)$); since this is an
  amalgamated coproduct that identifies elements of the epicentres,
  Theorems~\ref{thm:amalgamnotcentralprod}
  and~\ref{thm:epiforamalgams} yield that $G_2$ is not capable and lies
  in~$\mathcal{R}_p$ (the theorem of Ellis mentioned in the
  introduction guarantees the existence of a nontrivial relation among
  nontrivial commutators of any transversal).

  If $G$ is abelian, then let $H$ be the central product of
  $\amalgam{G}{C_p}{}$ with $\amalgam{C_p}{C_p}{}$ identifying a
  generator of $\amalgam{C_p}{C_p}{}$ with a nontrivial commutator in
  $\amalgam{G}{C_p}{}$; this is a non-capable group. We now let
  $G_2=\amalgam{H}{E}{\varphi}$ where $E$ is again the extraspecial
  group of order $p^5$ and exponent~$p$, and $\varphi$ identifies
  elements of the epicentres. Again, $G_2$ is not capable and lies in $\mathcal{R}_p$. 
  The rank inequalities are immediate. 
\end{proof}

\section*{References}
\bibliographystyle{amsalpha}

\begin{biblist}
\bib{baer}{article}{
  author={Baer, Reinhold},
  title={Groups with preassigned central and central quotient group},
  date={1938},
  journal={Trans. Amer. Math. Soc.},
  volume={44},
  pages={387\ndash 412},
}

\bib{beyl}{article}{
  author={Beyl, F.~Rudolf},
  author={Felgner, Ulrich},
  author={Schmid, Peter},
  title={On groups occurring as central factor groups},
  date={1979},
  journal={J. Algebra},
  volume={61},
  pages={161\ndash 177},
  review={\MR {81i:20034}},
}

\bib{ellis}{article}{
  author={Ellis, Graham},
  title={On the capability of groups},
  date={1998},
  journal={Proc. Edinburgh Math. Soc. (2)},
  volume={41},
  number={3},
  pages={487\ndash 495},
  review={\MR {2000e:20053}},
}

\bib{metab}{article}{
  author={Golovin, O.~N.},
  title={Metabelian products of groups},
  date={1956},
  journal={Amer. Math. Soc. Transl. Ser. 2},
  volume={2},
  pages={117\ndash 131},
  review={\MR {17:824b}},
}

\bib{hallpgroups}{article}{
    author={Hall, P.},
     title={The classification of prime-power groups},
      date={1940},
   journal={J. Reine Angew. Math.},
    volume={182},
     pages={130\ndash 141},
    review={\MR{2,211b}},
}


\bib{heinnikolova}{article}{
  author={Heineken, Hermann},
  author={Nikolova, Daniela},
  title={Class two nilpotent capable groups},
  date={1996},
  journal={Bull. Austral. Math. Soc.},
  volume={54},
  number={2},
  pages={347\ndash 352},
  review={\MR {97m:20043}},
}

\bib{machenry}{article}{
  author={Mac{H}enry, T.},
  title={The tensor product and the 2nd nilpotent product of groups},
  date={1960},
  journal={Math. Z.},
  volume={73},
  pages={134\ndash 145},
  review={\MR {22:11027a}},
}

\bib{expp}{article}{
  author={Magidin, Arturo},
  title={On the capability of finite groups of class two and prime exponent},
  eprint={arXiv:0708.2391 (math.GR)},
}



\end{biblist}

\end{document}